\documentclass[a4paper, 10pt]{amsart}

\usepackage{enumitem}
\usepackage{amsmath}
\usepackage{fancybox, graphicx}
\usepackage[usenames, dvipsnames]{color}
\usepackage[colorlinks=true,
        raiselinks=true,
        linkcolor=MidnightBlue,
        citecolor=Mahogany,
        urlcolor=ForestGreen,
        pdfauthor=Debasish Chatterjee,
        pdftitle={Attaining bounded variance of a noisy linear system with bounded control inputs},
        pdfkeywords={stochastic reachability and avoidance, Markov processes},
        pdfsubject={Technical Report},
        plainpages=false]{hyperref}
\usepackage{mathtools}
\usepackage[square, sort&compress, authoryear]{natbib}
\usepackage[bitstream-charter]{mathdesign}
\DeclareSymbolFont{usualmathcal}{OMS}{cmsy}{m}{n}
\DeclareSymbolFontAlphabet{\mathcal}{usualmathcal}

\DeclareMathOperator{\trace}{tr}
\DeclareMathOperator{\rank}{rank}

\newcommand{\reachab}{{\mathcal R}}
\newcommand{\lnorm}{\left\lVert}
\newcommand{\rnorm}{\right\rVert}
\newcommand{\diag}{\mathop{\rm diag}}
\renewcommand{\le}{\leq}

\newcommand{\Let}{\coloneqq}
\newcommand{\teL}{\eqqcolon}

\newcommand{\ball}{{\mathcal B}}

\def\bmat{\left[ \begin{array}}
\def\emat{\end{array} \right]}

\DeclareMathOperator{\sat}{sat}

\definecolor{mred}{rgb}{0.6, 0, 0}
\definecolor{mgreen}{rgb}{0, 0.5, 0}
\definecolor{mblue}{rgb}{0, 0, 0.5}
\definecolor{mcyan}{rgb}{0, 0.5, 0.5}

\newcommand{\R}{\ensuremath{\mathbb{R}}}
\newcommand{\N}{\ensuremath{\mathbb{N}}}
\newcommand{\Nz}{\ensuremath{\mathbb{N}_0}}

\newcommand{\lra}{\ensuremath{\longrightarrow}}

\newcommand{\fa}{\ensuremath{\forall\,}}
\renewcommand{\le}{\ensuremath{\leqslant}}

\renewcommand{\mapsto}{\ensuremath{\longmapsto}}

\newcommand{\EE}{\ensuremath{\mathsf{E}}}
\newcommand{\PP}{\ensuremath{\mathsf{P}}}

\newcommand{\norm}[1]{\ensuremath{\left\lVert #1 \right\rVert}}

\newcommand{\abs}[1]{\ensuremath{\left\lvert{#1}\right\rvert}}

\newcommand{\secref}[1]{\S\ref{#1}}

\newcommand{\transp}{\ensuremath{^{\scriptscriptstyle{\mathrm T}}}}
\newcommand{\sigalg}{\ensuremath{\mathfrak{F}}}

\newcommand{\RemarkEnd}{\hspace{\stretch{1}}{$\vartriangleleft$}}

\newcommand{\AssumptionEnd}{\hspace{\stretch{1}}{$\diamondsuit$}}

\swapnumbers
\newtheoremstyle{nonum}{4pt}{4pt}{}{}{\itshape}{.}{ }{\thmname{#1}\thmnote{ (\mdseries #3)}}
\theoremstyle{nonum}

\newtheoremstyle{nonumt}{4pt}{4pt}{\slshape}{}{\bfseries}{.}{ }{\thmname{#1}\thmnote{ (\mdseries #3)}}
\theoremstyle{nonumt}

\numberwithin{equation}{section}
\newtheoremstyle{dcstyle}{4pt}{4pt}{\slshape}{}{\bfseries}{.}{ }{}

\theoremstyle{dcstyle}
\newtheorem{theorem}[equation]{Theorem}

\newtheorem{lemma}[equation]{Lemma}
\newtheorem{proposition}[equation]{Proposition}

\theoremstyle{definition}

\theoremstyle{remark}
\newtheorem{remark}[equation]{Remark}

\newtheorem{assumption}[equation]{Assumption}

\makeatletter
\def\tagform@#1{\maketag@@@{\ignorespaces#1\unskip\@@italiccorr}}
\makeatother

\numberwithin{figure}{section}

\allowdisplaybreaks

\title[Mean-square boundedness of linear systems with bounded control]{Attaining mean square boundedness of a marginally stable stochastic linear system with a bounded control input}

\author[F.~Ramponi]{Federico Ramponi}
\author[D.~Chatterjee]{Debasish Chatterjee}
\author[A.~Milias-Argeitis]{Andreas Milias-Argeitis}
\author[P.~Hokayem]{Peter Hokayem}
\author[J.~Lygeros]{John Lygeros}

\thanks{This research was partially supported by the Swiss National Science Foundation, grant 200021-122072.}
\address{Automatic Control Laboratory, ETL I28, ETH Z\"urich, Physikstrasse 3, 8092 Z\"urich, Switzerland}
\urladdr{\url{http://control.ee.ethz.ch}}
\email{\{ramponif,chatterjee,milias,hokayem,lygeros\}@control.ee.ethz.ch}

\begin{document}

\begin{abstract}
In this article we construct control policies that ensure bounded variance of a noisy marginally stable linear system in closed-loop.  It is assumed that the noise sequence is a mutually independent sequence of random vectors, enters the dynamics affinely, and has bounded fourth moment.  The magnitude of the control is required to be of the order of the first moment of the noise, and the policies we obtain are simple and computable.
\end{abstract}

\maketitle

\section{Introduction}

Stabilization of stochastic linear systems with bounded control inputs has attracted considerable attention over the years. This is due to the fact that incorporating bounds on the control is of paramount importance in practical applications; suboptimal control strategies such as receding-horizon control \citep{ref:ChaHokLyg-09,ref:HokChaLyg-09}, and rollout algorithms \citep{ref:bertsekasDP1}, among others, were designed to incorporate such constraints with relative ease, and have become widespread in applications. However, the following question remains open: \emph{when is a linear system with possibly unbounded additive stochastic noise globally stabilizable with bounded inputs?} In this article we shall provide sufficient conditions that give a positive answer to this question with minimal hypotheses.

Bounded input control has a rich and important history in the control literature \citep{ref:sussmannBoundedControls92, ref:sussmannBoundedControls94, ref:sussmannBoundedControlsFinalDT, ref:stoorvogel96, ref:stoorvogelACC07Incorrect}. The deterministic version of the bounded input stabilization problem was solved completely in a series of articles~\citep{ref:sussmannBoundedControls92, ref:sussmannBoundedControls94} culminating in~\citep{ref:sussmannBoundedControlsFinalDT}.  It was demonstrated in \citep{ref:sussmannBoundedControlsFinalDT} that global asymptotic stabilization of a discrete-time linear system
\begin{equation}
\label{e:detsys}
	x_{t+1} = A x_t + B u_t\tag{$(\star)$}
\end{equation}
with bounded feedback inputs is possible if and only if the transition matrix has spectral radius at most $1$, and the pair $(A, B)$ is stabilizable with arbitrary controls. Moreover, extensions to the output feedback case have appeared in \citep{ref:BaoLinSon-00,ref:ChiLin-03}.

In the presence of affine stochastic noise the linear system \eqref{e:detsys} becomes
$x_{t+1} = A x_t + B u_t + w_t$,
where $(w_t)_{t\in\Nz}$ is a collection of independent (but not necessarily identically distributed) random vectors in $\R^d$ with possibly inter-dependent components at each time $t$.  With an arbitrary noise it is clearly not possible to ensure mean-square boundedness; for instance, if the noise has a spherically symmetric Cauchy distribution on $\R^d$, then given any initial condition $x_0\in\R^d$, the second moment of $x_1$ does not even exist.  Similarly, if the second moment of the noise becomes unbounded with time, it is not possible to control the second moment of the process $(x_t)_{t\in\Nz}$.  It is necessary to assume, at least, that the noise has bounded variance.

Going beyond this necessary condition, it is not difficult to establish mean-square boundedness of such a system with bounded controls under the assumption that $A$ is Schur stable, i.e., all eigenvalues of $A$ are contained in the interior of the unit disk (the proof of this fact relies on standard Foster-Lyapunov techniques \citep{ref:meyn93}). However, to the best of our knowledge, there is no proof that the same can be ensured for a marginally stable linear system.  Results in this direction were reported in~\citep{ref:stoorvogelACC07Incorrect}, but to the best of our understanding \emph{conclusive} proofs of the facts reported in the present article are still missing in the literature.


%

In this article, we develop easily computable bounded control policies for the case when $A$ is marginally stable
and $(A,B)$ is stabilizable.
Our policy is not anyway stationary and is in general
chosen from the class of finite $k$-history-dependent and/or non-stationary policies. With respect to the case when $A$ is orthogonal, it turns out that if the system is reachable in one step (i.e., $\rank B = $ the dimension of the state space), we do get stationary feedback policies.  In the more general case when the system \eqref{e:detsys} is reachable in $k$ steps (with arbitrary controls), we propose a feedback policy for a sub-sampled system derived from the original one, which, for the actual system, turns out to be a $k$-history-dependent policy. In fact, in this case we realize our policy as successive concatenations of a fixed $k$-length policy.
In the most general situation
we propose a $k$-history-dependent policy, where $k$
is now the reachability index of the particular subsystem of $(A,B)$
for which the dynamics matrix is orthogonal.
In all the mentioned cases, the length of the policy
is at most equal to the dimension of the state space;
memory requirements for even the most general case are, therefore, modest.

Note that in our setting we do \emph{not} assume that the noise is white.  For our purposes the requirements on the noise are rather general, namely, the fourth moment of the noise should be uniformly bounded, and the noise vectors should be independent of each other (identical distribution at each time is not assumed).  In particular, we do \emph{not} assume Gaussian structure of the noise.  It turns out that to ensure stabilization we need the controller to be sufficiently strong, in the sense that the control input norm bound should be bigger than a uniform bound on the first moment of the noise.

Section \ref{SECTION_MAIN} contains a precise statement of our result in the most general hypotheses ($A$ marginally stable and $(A,B)$ stabilizable), and a brief sketch of the proof.  In Section \ref{SECTION_PROOF}, after some preliminary material, we prove the attainability of bounded second moment for a random walk, then we generalize the result under weaker and weaker hypotheses, finally culminating in the proof of the main theorem of Section \ref{SECTION_MAIN}. Section \ref{SECTION_EXAMPLE} presents a numerical example illustrating our results, and Section \ref{SECTION_CONCLUSION} concludes the article with a conjecture.

\section{Main result}
\label{SECTION_MAIN}

\subsection{Statement of the theorem}

Consider the discrete-time linear system
\begin{equation}
\label{STABLE_SYSTEM}
x_{t+1} = A x_t + B u_t + w_t,\qquad x_0 = x,\quad t\in\Nz,
\end{equation}
where the following hold: $x\in\R^d$ is given;
the state $x_t$ at time $t$ takes values in $\R^d$; $A\in\R^{d\times d}$, all the eigenvalues of $A$
lie in the closed unit circle, and those eigenvalues $\lambda$ such that $|\lambda|=1$
have equal algebraic and geometric multiplicities; $B\in\R^{d\times m}$, and the control $u_t$ at time $t$ takes values in $\R^m$; $(w_t)_{t\in\Nz}$ is an $\R^d$-valued random process with mean zero and $\EE\bigl[w_t w_t\transp\bigr] = Q_t$.

%
%

Our objective is to synthesize a $k$-history-dependent control policy\footnote{See~\secref{s:prelims} for definitions of policies.} $\pi = (\pi_t)_{t\in\Nz}$, consisting of successive concatenations of $k$-length sequence $\tilde\pi_{0:k-1} \Let \bigl[\tilde\pi_0, \cdots, \tilde\pi_{k-1}\bigr]$ of maps, $\tilde\pi_i:\R^{d}\lra\R^m$ for $i = 0, \ldots, k-1$, such that $\pi_t:\R^{d\times k}\lra\R^m$ is measurable, $u_t \Let \pi_t\bigl(x_t, x_{t-1},\ldots, x_{t-k+1}\bigr)$, the sequence $(u_t)_{t\in\Nz}$ is bounded, and the state of the closed-loop system
\begin{equation}
\label{e:syscl}
x_{t+1} = A x_t + B \pi_t\bigl(x_t, x_{t-1}, \ldots, x_{t-k+1}\bigr) + w_t, \qquad x_0 = x, \quad t\in\Nz,
\end{equation}
has bounded second-order moment. (To simplify the notation, we fix $x_{-k+1} = \cdots = x_{-1} = x_0$.)
The following is our main result:
\begin{theorem}
\label{MAIN_THEOREM}
Consider the system \eqref{STABLE_SYSTEM}. Suppose that
the pair $(A, B)$ is {\em stabilizable}, and that
$\sup_{t\in\Nz}\EE\bigl[\norm{w_t}^4\bigr] < \infty$.
Then there exist an $R > 0$ and a deterministic $k$-history-dependent policy
$(\pi_t)_{t\in\Nz}$,
with $k\leq d$ and $\norm{\pi_t(\cdot)} \le R$ for every $t$,
such that
\begin{enumerate}[label={\rm (P\arabic*)}, leftmargin=*, align=right]
\item for every fixed $x\in\R^d$ the process $(x_t)_{t\in\Nz}$ that solves the recursion \eqref{e:syscl}
satisfies $\sup_{t\in\Nz}\EE_x\bigl[\norm{x_t}^2\bigr] < \infty$, and
\item in the absence of the random noise the origin is asymptotically stable
for the closed-loop system.
\end{enumerate}
%
\end{theorem}


\subsection{Sketch of the proof}
Our proof is built in a series of steps, moving from simpler to progressively more complex systems. The starting point is the $d$-dimensional random walk $x_{t+1} = x_t + u_t + w_t$. In this case we employ the main result of \citep{ref:pemantle99} to design a policy that guarantees mean-square boundedness of the closed-loop system. We then consider the system $x_{t+1} = A x_t + Bu_t + w_t$, where $u_t$ is a $d$-dimensional control input, $\rank B = d$, and $A$ is orthogonal. With the help of a time-varying injective linear transformation this case is reduced to the $d$-dimensional random walk. The third case that we consider is that of the system $x_{t+1} = A x_t + B u_t + w_t$, where $u_t\in\R^m$ and $A$ is orthogonal.  This is reduced to the second case above with the aid of an injective linear transformation derived from the reachability matrix of the pair $(A, B)$ (recall that by assumption the reachability matrix has rank $d$). Finally, the general case when $A$ is just stable and $(A,B)$ stabilizable is reduced to the third case with the observation that, in view of the stability hypothesis, $A$ acts as an orthogonal map on its invariant subspace that corresponds to the eigenvalues that lie on the unit circle.

Arguments for establishing mean-square boundedness of stochastic dynamical systems typically rely on $L_1$-bounded-ness of a Lyapunov-like functional of the system. The latter can be established in at least three different ways: The first is via the classical Foster-Lyapunov drift-conditions \citep{ref:foss04,ref:meyn93} and its various refinements; the second is via excursion-theoretic analysis~\citep{ref:palExcur} that relies primarily on the existence of certain supermartingales as long as the process is outside some bounded set; the third is via martingale inequalities~\citep{ref:pemantle99}, which applies to more general scalar-valued processes than Markov processes, and in the presence of bounded controls, provides the basic machinery for establishing our Theorem \ref{MAIN_THEOREM}.

\section{Proof of the main result}
\label{SECTION_PROOF}

\subsection{Preliminaries}

        \label{s:prelims}
            Let $\Nz$ be the set of nonnegative integers $\{0, 1, 2, \ldots\}$. The standard $2$-norm on Euclidean spaces is denoted by $\norm{\cdot}$ and the absolute value on $\R$ by $\abs{\cdot}$. In a Euclidean space we denote by $\ball_r$ the closed Euclidean ball of radius $r$ centered at the origin. If $(y_t)_{t\in\Nz}$ is a random process on a probability space $(\Omega, \sigalg, \PP)$, taking values in some Euclidean space, we let $\EE_x[\varphi(y_s; s=0, 1,\ldots, t)]$ denote the conditional expectation of a measurable mapping $\varphi$ of the process up to time $t$, given the initial condition $y_0 = x$; in particular we define the $n$-th moment of $y_t$ as $\EE_x[\norm{y_t}^n]$.  We denote conditional expectation given a sub-$\sigma$-algebra $\sigalg'$ of $\sigalg$ as $\EE[\cdot\,|\,\sigalg']$. For $r > 0$ let $\sat_r:\R^d \lra \ball_r$ be defined by $\sat_r(y) \Let y$ if $y\in\ball_r$ and $\sat_r(y) \Let r y/\norm y$ otherwise. 
			Note that $\sat_r(\cdot)$ is \emph{not} the component-wise saturation function. Given matrices $A\in \R^{d\times d}$ and $B\in\R^{d\times m}$ we define the $k$-step reachability matrix $\reachab_k \Let \bmat{cccc} B & AB & \cdots & A^{k-1}B \emat$.

            We specialize the general definition of a policy~\citep[Chapter~2]{ref:hernandez-lerma1996} to our setting. A policy $\pi \Let (\pi_t)_{t\in\Nz}$ is a sequence of measurable maps $\pi_t:\R^{d\times k}\lra\R^m$ for some $k\in\N$, such that the control at time $t$ is $\pi_t\bigl(x_t, x_{t-1}, \ldots, x_{t-k+1}\bigr)$. The policy $\pi = (\pi_t)_{t\in\Nz}$ we have defined is also known as a {\em deterministic $k$-history-dependent policy} in the literature. A special case of these policies is a {\em deterministic feedback policy} or simply a {\em feedback} if $k = 1$ in the definition of a deterministic history-dependent policy. Under deterministic feedback policies the closed-loop system is Markovian~\citep[Proposition~2.3.5]{ref:hernandez-lerma1996}. A further special case is when $\pi_t = f$, a fixed measurable mapping $f:\R^d\lra\R^m$ for $t\in\Nz$; this is known as a {\em stationary feedback policy}.



\begin{lemma}
\label{LEMMA_BASE}
Let $B_1, \cdots, B_k$ be $d\times m$ matrices,
$M \Let \bmat{ccc} B_1 & \cdots & B_k \emat$,
and $\sigma_d$
denote the minimum singular value of $M$.
If $\rank M = d$, then for all $r>0$ every vector
$v \in \R^d$ belonging to $\ball_r$ can be expressed as
$v = \sum_{i=1}^k B_i u_i$, with $u_i \in \R^m$ and
$\lnorm u_i \rnorm \leq r\sigma_d^{-1}$.
In particular, if $B\in \R^{d\times d}$ and $\rank B = d$,
then every vector $v \in \R^d$ belonging to $\ball_r$ can be expressed as
$v = B u$, where $u \in \R^d$, $\lnorm u \rnorm \leq r\sigma_d^{-1}$.
\end{lemma}

\begin{proof}
$\rank M = d$ implies that $km \geq d$.
Hence,
$M = \bmat{ccc} B_1 & \cdots & B_k \emat \in \R^{d\times km}$ is a ``flat'' matrix.
Let
$
    M = USV\transp = U \bmat{cc} \Sigma & 0 \\ \emat V\transp
$
be a singular value decomposition of $M$, where $\Sigma = \diag(\sigma_1, ..., \sigma_d)$.
Since $M$ has full rank, the matrix $\Sigma$ is invertible.
Hence every vector $v \in \R^d$ can be expressed as $v = Mu$, where $u = M^+v$ and
$M^+ = V \bmat{c} \Sigma^{-1} \\0\\ \emat U\transp \in \R^{km\times d}$
is the
Moore-Penrose pseudoinverse of $M$.
Since $U, V$ are orthogonal, for any $\rho>0$ we have
$
\inf_{\lnorm u \rnorm = \rho}  \lnorm M u \rnorm
= \inf_{\lnorm V\transp u \rnorm = \rho}  \lnorm U \bmat{cc} \Sigma & 0\\ \emat V\transp u \rnorm
= \inf_{\lnorm \upsilon \rnorm = \rho}  \lnorm \Sigma \upsilon \rnorm
= \rho \sigma_d.
$
%
Hence, the image of $\ball_\rho$ under $M$ contains
$\ball_{\rho \sigma_d}$, and if we choose
$\rho = r \sigma_d^{-1}$, then the image
of $\ball_\rho$ under $M$ contains $\ball_r$.
Notice that $\sigma_d^{-1}$ is also the greatest singular value of
$M^+$, and indeed we have
$
\sup_{\lnorm v \rnorm = r}  \lnorm M^+v \rnorm
= \sup_{\lnorm U\transp v \rnorm = r}  \lnorm V \bmat{c} \Sigma^{-1} \\ 0\\ \emat U\transp v \rnorm
= \sup_{\lnorm \nu \rnorm = r}  \bigl\|\Sigma^{-1} \nu \bigr\|
= r \sigma_d^{-1}.
$
Summing up, every $v \in \ball_r$ can be expressed as
$v = Mu$, where $u\in \R^{k m}$ and $\lnorm u \rnorm \leq r\sigma_d^{-1}$.
It remains to notice that $u$ can be partitioned according to the
partition of $M$, that is
$
v = Mu =
\bmat{cccc} B_1 & B_2 & \cdots & B_k \emat
\begin{bmatrix}u_1\transp & \cdots & u_k\transp\end{bmatrix}\transp
 = \sum_{i=1}^k B_i u_i
$
and the bound $\lnorm u \rnorm \leq r\sigma_d^{-1}$ implies
$\lnorm u_i \rnorm \leq r\sigma_d^{-1}$ for all $i = 1\cdots k$.
%
\end{proof}

        \subsection{The $d$-dimensional random walk}
        \label{s:drw}
            At the core of our proof is the $d$-dimensional random walk:
            \begin{equation}
            \label{e:sysdrw}
                x_{t+1} = x_t + u_t + w_t,\qquad x_0 = x, \quad t\in\Nz,
            \end{equation}
            with the state $x_t\in\R^d$, the control $u_t\in\R^d$ with $\norm{u_t} \le r$ for some $r > 0$, the noise process $(w_t)_{t\in\Nz}$ satisfies the following assumption:

			\begin{assumption}\mbox{}
			\label{a:w}
                \begin{itemize}[label=$\diamond$, leftmargin=*]
                    \item $(w_t)_{t\in\Nz}$ are mutually independent $d$-dimensional random vectors (not necessarily identically distributed),
    				\item $\EE[w_t] = 0$, $\EE\bigl[w_tw_t\transp\bigr] = Q_t$ for all $t\in\Nz$,
                    \item there exist $C_4 > 0$ such that $\EE\bigl[\norm{w_t}^4\bigr] \le C_4$ for all $t\in\Nz$.\AssumptionEnd
                \end{itemize}
			\end{assumption}

			Let $C_1 \Let \sup_{t\in\Nz}\EE\bigl[\norm{w_t}\bigr]$; this is well-defined because by Jensen's inequality we have $C_1 \le \sqrt[4]{C_4}$. Let $(\sigalg_t)_{t\in\Nz}$ be the natural filtration of the system~\eqref{e:sysdrw}. Our proof of Theorem~\ref{MAIN_THEOREM} relies on the following (immediate) adaptation of the fundamental result \citep[Theorem~1]{ref:pemantle99}.

            \begin{proposition}
            \label{p:pemantle99}
                Let $(\xi_t)_{t\in\Nz}$ be a sequence of nonnegative random variables on some probability space $(\Omega, \sigalg, \PP)$, and let $(\sigalg_t)_{t\in\Nz}$ be any filtration to which $(\xi_t)_{t\in\Nz}$ is adapted. Suppose that there exist constants $b > 0$, and $J, M < \infty$, such that $\xi_0\le J$, and for all $t$:
                \begin{gather}
                    \EE\bigl[\xi_{t+1} - \xi_t\big|\sigalg_t\bigr] \le -b\quad \text{on the event }\{\xi_t > J\},\quad\text{and}\label{e:FLcond}\\
                    \EE\bigl[\abs{\xi_{t+1} - \xi_t}^4\big|\xi_0,\ldots, \xi_t\bigr] \le M.\label{e:pcond}
                \end{gather}
                Then there exists a constant $c = c(b, J, M) > 0$ such that $\displaystyle{\sup_{t\in\Nz}\EE\bigl[\xi_t^{2}\bigr] \le c}$.
            \end{proposition}

            \begin{lemma}
            \label{l:FLcond}
                Consider the system \eqref{e:sysdrw}, and define $\xi_t \Let \norm{x_t}$, $t\in\Nz$. There exists a constant $b > 0$, such that for any $r > C_1$ condition~\eqref{e:FLcond} holds in closed-loop with the control $u_t = -\sat_r(x_t)$.
            \end{lemma}

            \begin{proof}
                Fix $t\in\Nz$ and $r > C_1$. We have 
                \begin{align*}
                    \EE\bigl[\xi_{t+1} - \xi_t\big|\sigalg_t\bigr] & = \EE\bigl[\norm{x_{t+1}} - \norm{x_t}\big|\sigalg_t\bigr] = \EE\bigl[\bigl\|x_t + u_t + w_t\bigr\| - \norm{x_t}\big|\sigalg_t\bigr]\\
                    & = \EE\bigl[\norm{x_t - \sat_r(x_t) + w_t} - \norm{x_t}\big|\sigalg_t\bigr]\\
                    & \le \EE\bigl[\norm{x_t - \sat_r(x_t)} + \norm{w_t} - \norm{x_t}\big|\sigalg_t\bigr].
                \end{align*}
                Let $J = r$ and $b\Let r - C_1$. On the set $\{\norm{x_t} > J\}$ we have $\norm{x_t - \sat_r(x_t)} - \norm{x_t} = -r$. From the above we get, on the set $\{\norm{x_t} > J\}$,
                \begin{align*}
                    \EE\bigl[\xi_{t+1} - \xi_t\big|\sigalg_t\bigr] & \le \EE\bigl[\norm{x_t - \sat_r(x_t)} + \norm{w_t} - \norm{x_t}\big|\sigalg_t\bigr]\\
                    & = -r + \EE\bigl[\norm{w_t}\bigr]\\
                    & \le -b,
                \end{align*}
                where $b$ is positive by our hypothesis. The assertion follows.
            \end{proof}

            \begin{lemma}
            \label{l:pcond}
                Consider the system~\eqref{e:sysdrw} and define $\xi_t \Let \norm{x_t}$, $t\in\Nz$. Then for the closed-loop system with $u_t = -\sat_r(x_t)$ there exists a constant $M = M(C_4) > 0$ such that~\eqref{e:pcond} holds.
            \end{lemma}

            \begin{proof}
                Fix $r > C_1$. Applying the triangle inequality successively, we have
                \[ 
                    \abs{\xi_{t+1} - \xi_t}^4 = \abs{\norm{x_{t+1}} - \norm{x_t}}^4 \le \norm{x_{t+1} - x_t}^4 = \norm{u_t + w_t}^4 \le \bigl(r + \norm{w_t}\bigr)^4,
                \]
                which leads to
                \[
                    \EE\bigl[\abs{\xi_{t+1} - \xi_t}^4\,\big|\,\xi_0, \ldots, \xi_t\bigr] \le \EE\bigl[\bigl(r + \norm{w_t}\bigr)^4\big|\xi_0, \ldots, \xi_t\bigr] = \EE\bigl[\bigl(r + \norm{w_t}\bigr)^4\bigr].
                \]
                Since the fourth moment of $w_t$ is uniformly bounded, expanding the right-hand side above and applying Jensen's inequality shows that there exists some $M = M(C_4) > 0$ such that $\EE\bigl[\bigl(r + \norm{w_t}\bigr)^4\bigr] \le M$. The assertion follows.
            \end{proof}

            \begin{proposition}
            \label{PROPOSITION_RW}
                For $r > 0$ consider the system~\eqref{e:sysdrw} under the deterministic stationary feedback policy $u_t = -\sat_r(x_t)$:
                \begin{equation}
				\label{e:randomwalkclosedloop}
                    x_{t+1} = x_t -\sat_r(x_t) + w_t,\qquad x_0 = x, \quad t\in\Nz.
                \end{equation}
 				Then for every $r > C_1$ the system~\eqref{e:randomwalkclosedloop} satisfies $\sup_{t\in\Nz}\EE_x\bigl[\norm{x_t}^2\bigr] \le c$ for some $c = c(x, C_1) < \infty$.
            \end{proposition}

            \begin{proof}
                Let $r = C_1 + b$ for some $b > 0$ and $J \Let \max\bigl\{r, \norm{x}\bigr\}$. Lemma~\ref{l:FLcond} guarantees that~\eqref{e:FLcond} holds, and Lemma~\ref{l:pcond} shows that there exists an $M > 0$ such that~\eqref{e:pcond} holds. The assertion now is an immediate consequence of Proposition~\ref{p:pemantle99}.
            \end{proof}

\subsection{The case of $A$ orthogonal}
\label{SUBSEC_ORTHOGONAL}

Next we establish part (P1) of the main theorem in the particular case of $A$ being {\em orthogonal}.

\begin{lemma}
\label{LEMMA_RANDOMWALK_POLICY}
Consider the system
$
y_{t+1} = Ay_t + u_t + w_t
$,
where $y_t$ and $u_t$ take values in $\R^d$,
$A$ is orthogonal,
and $(w_t)_{t\in\Nz}$ satisfies Assumption~{\rm \ref{a:w}}.
There exist a constant $r > 0$ and a deterministic stationary policy
$\pi = (f, f, \cdots)$ such that $\norm{f(y)} \le r$ for all $y\in\R^d$
and $t\in\Nz$,
and the closed-loop system
\begin{equation}
\label{EQ_CLOSEDLOOP_Y}
y_{t+1} = Ay_t + f(y_t) + w_t
\end{equation}
under this policy satisfies $\sup_{t\in\Nz}\EE_x\bigl[ \lnorm y_t \rnorm^2\bigr] < \infty$.
\end{lemma}

\begin{proof}
Consider the process $(z_t)_{t\in\Nz}$ defined by $z_t \Let (A\transp)^t \ y_t$.
%
%
The second moment of $z_t$ is the same as that of $y_t$ due to orthogonality of $A$:
\[
\EE_x \bigl[ \lnorm z_t \rnorm^2 \bigr]
= \EE_x \bigl[ \lnorm (A\transp)^t \ y_t \rnorm^2 \bigr]
= \EE_x\bigl[ y_t\transp A^t (A\transp)^t y_t \bigr]
= \EE_x\bigl[ y_t\transp y_t\bigr]
= \EE_x \bigl[ \lnorm y_t \rnorm^2 \bigr].
\]
Now we have
\begin{equation}
\label{EQ_TRICK}
\begin{split}
z_{t+1}
&= (A\transp)^{t+1} \ y_{t+1}
= (A\transp)^{t} \ y_t + (A\transp)^{t+1} \ u_t + (A\transp)^{t+1} \ w_t = z_t + \bar{u}_t + \bar{w}_t,
\end{split}
\end{equation}
where the mapping $u_t \mapsto \bar{u}_t \Let (A\transp)^{t+1} \ u_t$ is isometric and invertible, and $(\bar w_t)_{t\in\Nz}$ defined by $\bar{w}_t \Let (A\transp)^{t+1} \ w_t$, is a sequence of zero-mean, independent (although in general {\em not} identically distributed) random vectors, with fourth moment given by
$
\EE \bigl[ \lnorm \bar{w}_t \rnorm^4 \bigr]
= \EE \bigl[ \bigl\|(A\transp)^{t+1} \ w_t \bigr\|^4 \bigr]
= \EE \bigl[ \lnorm w_t \rnorm^4 \bigr]
\le C_4.
$
Due to Proposition \ref{PROPOSITION_RW}, there exists a constant $r$ such that the closed-loop system \eqref{EQ_TRICK} under the policy $\bar u_t = -\sat_r(z_t) \teL \bar f(z_t)$ has bounded second moment. Consequently, the original system \eqref{EQ_CLOSEDLOOP_Y} has bounded second moment under the policy
%
\[
u_t = A^{t+1} \bar u_t = A^{t+1} \bar f(z_t) = -A^{t+1} \sat_r\left( (A\transp)^t \ y_t \right) \teL f_t(y_t).
\]
%
Noting that for any orthogonal matrix $A$
we have $\sat_r(Ay) = A \sat_r(y)$, we arrive at
$
u_t = f_t(y_t) = - A \sat_r (y_t) \teL f(y_t),
$
which is indeed a stationary feedback. Moreover, since $\norm{A\sat_r (y_t)}\leq r$, we have $\lnorm f(y_t) \rnorm \leq r$.
\end{proof}

In the following we will consider a nonstationary policy
obtained by successive concatenations of a $k$-length policy $(f_0, f_1, \cdots f_{k-1})$
acting on the ``sub-sampled'' process $(x_{nk})_{n\in\Nz}$. More precisely, our policy has the form
$
u_t = Bf_{t\ {\bf mod}\ k}(x_{(t \div k)k})
$
where the ``$\div$'' symbol denotes integer division and ``${\bf mod}$'' its remainder.
In words, we break the time line into segments of length $k$, and within each segment
we let the controls be given by $f_0, f_1, \cdots f_{k-1}$, applied in this order always to the
first state observed in the segment.
For example,
$x_1 = x_0 + Bf_0(x_0) + w_0$,\
$x_2 = x_1 + Bf_1(x_0) + w_1$,\
...,\
$x_k = x_{k-1} + Bf_{k-1}(x_0) + w_{k-1}$,\
$x_{k+1} = x_{k} + Bf_0(x_{k}) + w_{k}$,\
$x_{k+2} = x_{k+1} + Bf_1(x_{k}) + w_{k+1}$,\
and so on.
%

\begin{lemma}
\label{LEMMA_FINAL}
Consider the system
\begin{equation}
\label{EQ_SYSTEM_WITH_B}
x_{t+1} = Ax_t + Bu_t + w_t,
\end{equation}
where $x_t$ takes values in $\R^d$, $u_t$ takes values in $\R^m$,
$A$ is orthogonal, the pair $(A, B)$ is reachable in $k$ steps
(i.e., $\rank \reachab_k = d$, where $\reachab_k = \bmat{cccc}B & AB & \cdots & A^{k-1} B\emat$),
and $(w_t)_{t\in\Nz}$ satisfies Assumption~{\rm \ref{a:w}}.
Then there exist a constant $\rho>0$ and a policy
$\pi = (f_0, f_1, \cdots f_{k-1}, f_0, f_1, \cdots)$
such that $\lnorm f_i(x) \rnorm \leq \rho$ for all $x\in\R^d$,
and the closed-loop system
\begin{equation}
\label{EQ_FINAL_POLICY}
x_{t+1} = Ax_t + Bf_{t\ {\bf mod}\ k}(x_{(t \div k)k}) + w_t
\end{equation}
under this policy satisfies $\sup_{t\in\Nz}\EE_x\bigl[ \lnorm x_t \rnorm^2\bigr] < \infty$.
\end{lemma}


\begin{proof}
Let $\tau\in\Nz$ and consider the evolution of \eqref{EQ_SYSTEM_WITH_B} from time $\tau k$ to time $(\tau+1)k$:
\begin{equation}
\label{EQ_SUBSAMPLED}
\begin{split}
x_{(\tau+1)k}
&= A^{k}\ x_{\tau k} +
\reachab_k
\begin{bmatrix}
u_{(\tau+1)k-1}\\ \vdots \\ u_{\tau k}
\end{bmatrix}
+ \sum_{i=0}^{k-1} A^{k-1-i} w_{\tau k + i}  = \bar{A} x_{\tau k} + \bar{u}_{\tau} + \tilde w_{\tau},
\end{split}
\end{equation}
where $\tilde w_\tau \Let \sum_{i=0}^{k-1} A^{k-1-i} w_{\tau k + i}$
is a random vector with mean zero and bounded fourth moment.
Since $\reachab_k$ has full rank, Lemma \ref{LEMMA_BASE} implies that for arbitrary $r>0$,
any $\bar{u}_\tau$ in $\ball_r$ can be expressed as $\bar{u}_\tau = \sum_{i=0}^{k-1} A^{k-1-i} B u_{\tau k + i}$,
where $\lnorm u_{\tau k + i} \rnorm\leq r\sigma_d^{-1}$ and
$\sigma_d$ is the smallest singular value of $\reachab_k$.
%
%
But from Lemma \ref{LEMMA_RANDOMWALK_POLICY} we know that there exists a particular $r > 0$ such that,
under the stationary policy $\bar u_\tau = f(x_{\tau k}) = -\bar{A} \sat_r(x_{\tau k})$,
the ``sub-sampled'' system \eqref{EQ_SUBSAMPLED}
has bounded second moment,
and $\lnorm \bar{u}_\tau \rnorm \leq r$.
Therefore, if we choose $\rho = r\sigma_d^{-1}$,
there exists a constant $c = c(x, C_1, C_4) > 0$ such that $\sup_{\tau\in\Nz}\EE_x\bigl[\norm{x_{\tau k}}^2\bigr] \le c$.
It follows from the system dynamics that for $n = 0, \ldots, k-1$,
\begin{align*}
	\EE_x\bigl[\norm{x_{\tau k + n}}^2\bigr] & \le 2\bigl(c +  n^2 r^2 \sigma_1(B)^2\bigl) + k \max_{n=0, \ldots, k-1}\trace Q_{\tau k + n}\\
	& \le 2\bigl(c +  n^2 r^2 \sigma_1(B)^2\bigl) + k \sqrt{C_4},
\end{align*}
where the last step follows from Jensen's inequality. Since the right-hand side above constitutes a uniform bound, this proves the assertion.
\end{proof}

\begin{remark}
\label{REMARK_POLICY}
The actual policy for \eqref{EQ_SYSTEM_WITH_B} is
$
\begin{bmatrix}
u_{(\tau +1)k-1}\\ \vdots\\ u_{\tau k}
\end{bmatrix}
= - \reachab_k^+ \bar{A} \sat_r(x_{\tau k}).
$
The proof above shows that all the inputs
$u_{(\tau+1)k-1}, \cdots, u_{\tau k}$ can be computed at time $\tau k$ in order to
counteract the future effect of the current state, i.e.\ $\bar{A} x_{\tau k}$,
and ignoring the effect of the noise for the following $k$ steps.
In the particular case when $B \in \R^{d\times d}$ has full rank, $m=d$, and obviously $k = 1$,
the above policy is {\em stationary}, and in particular it has the form:
$
u_t = f(x_t) = - B^{-1} A \sat_r(x_t).
$
Once again we have $\lnorm u_t \rnorm \leq r\sigma_d^{-1}$, where
this time $\sigma_d$ is the smallest singular value of $B$.
\RemarkEnd
\end{remark}

\subsection{Proof of Theorem \ref{MAIN_THEOREM}}

\begin{proof} 
Consider the system \eqref{STABLE_SYSTEM}, with $(A,B)$ stabilizable and
$(w_t)_{t\in\Nz}$ with bounded fourth moment.
If $A$ is Schur stable (that is, all the eigenvalues of $A$ belong to the interior of the unit disk),
the system with zero input has bounded second moment and is
asymptotically stable, and there is nothing to prove.
Otherwise, there exists a change of base in the state-space
that brings the original pair $(A, B)$
to a new pair $\bigl(\tilde A, \tilde B\bigr)$, where $\tilde A$ is in real
Jordan form \citep[p. 150]{ref:hornjohnson}.
In particular, choosing a suitable ordering of the Jordan blocks,
we can ensure that the pair $\bigl(\tilde A, \tilde B\bigr)$ has the form
$\left(\bigl[\begin{smallmatrix}A_{11} & 0\\ 0 & A_{22}\end{smallmatrix}\bigr], \bigl[\begin{smallmatrix}B_1\\ B_2\end{smallmatrix}\bigr]\right)$,
where $A_{11}$ is Schur stable,
and $A_{22}$ has its eigenvalues on the unit circle.
Due to the stability hypothesis
(the algebraic and geometric multiplicities of the eigenvalues of $A_{22}$ are equal),
$A_{22}$ is therefore block-diagonal with elements on the diagonal being either
$\pm 1$ or $2\times 2$ rotation matrices. As a consequence, $A_{22}$ is orthogonal.
Moreover, since $(A,B)$ is stabilizable, the pair $(A_{22}, B_2)$ must be reachable
in a number of steps $k \leq d$ which depends on the dimension of $A_{22}$
and the structure of $(A_{22}, B_2)$,
since it contains precisely the modes of $A$ which are not asymptotically stable.
Summing up, we can reduce the original system $x_{t+1} = A x_t + B u_t + w_t$ to the form
$
	\Bigl[\begin{smallmatrix} x^{(1)}_{t+1}\\ x^{(2)}_{t+1}\end{smallmatrix}\Bigr]
	= \Bigl[\begin{smallmatrix} A_{11} x^{(1)}_t\\ A_{22} x^{(2)}_t\end{smallmatrix}\Bigr]
	+ \Bigl[\begin{smallmatrix} B_1\\ B_2\end{smallmatrix}\Bigr] u_t
	+ \Bigl[\begin{smallmatrix} w^{(1)}_t\\ w^{(2)}_t\end{smallmatrix}\Bigr],
$
where $A_{11}$ is Schur stable, $A_{22}$ is orthogonal, $(A_{22}, B_2)$ is reachable, and
$\Bigl(\Bigl[\begin{smallmatrix} w^{(1)}_t\\ w^{(2)}_t\end{smallmatrix}\Bigr]\Bigr)_{t\in\Nz}$ is derived from
$(w_t)_{t\in\Nz}$ by means of linear transformations.
We know that since $A_{11}$ is Schur stable,
the noise $\bigl(w^{(1)}_t\bigr)_{t\in\Nz}$ has bounded second moment,
and the control inputs $(u_t)_{t\in\Nz}$ are bounded, then the $x^{(1)}$ sub-system
is mean-square bounded under any Markovian control~\citep[\S4]{ref:ChaHokLyg-09}.
Therefore, if under some bounded policy the $x^{(2)}$ sub-system is mean-square bounded,
the original system will also be mean-square bounded under the same policy.
Thus, at least for the proof of (P1), it suffices to restrict
our attention to the subsystem described by the pair $\bigl(A_{22}, B_2\bigr)$.
Suppose that this subsystem is reachable in a certain number $k\leq d$ of steps.

%

The proof of (P1) coincides with the proof of Lemma \ref{LEMMA_FINAL},
where we obtain $\rho = r\sigma_d^{-1}$ for $r > C_1$ and $\sigma_d$ is the smallest singular value of $\reachab_k$.
(Here, $\reachab_k = \bmat{cccc}B_2 & A_{22} B_2 & \cdots & A_{22}^{k-1} B_2\emat$.)
As the control authority required in the claim of the theorem, we choose precisely $R=\rho$.

To prove (P2), notice that for the closed-loop ``sub-sampled'' system without noise
under the policy $u_t = - \reachab_k^{+}\bar{A}\sat_r\bigl(x_t^{(2)}\bigr)$,
where $\bar{A} = A_{22}^k$,
it holds:
\begin{equation}
\label{EQ_SUBSAMPLED2}
x_{(\tau+1)k}^{(2)} = \bar{A} x_{\tau k}^{(2)} - \bar{A} \sat_r\bigl(x_{\tau k}^{(2)}\bigr).
\end{equation}
As long as $x_{\tau k}^{(2)}$ is outside $\ball_r$, $\lnorm x_{(\tau+1)k}^{(2)} \rnorm = \lnorm x_{\tau k}^{(2)} \rnorm - r$.
Hence, in a finite number of steps it must hold $\lnorm x_{\tau k}^{(2)} \rnorm <r$.
When for some $\bar{\tau}$ we have $\lnorm x_{(\bar{\tau}-1) k}^{(2)} \rnorm < r$, by the definition of
$\sat_r(\cdot)$ we have $x_{\bar{\tau} k}^{(2)} = 0$, and consequently $x_{\tau k}^{(2)} = 0$ for all $\tau \geq \bar{\tau}$.
Hence, the state of the closed-loop ``sub-sampled'' system converges to zero in {\em finite time}
for any initial condition.
Then, according to the chosen policy, for all $\tau \geq \bar{\tau}$ we have
$
\begin{bmatrix}
u_{(\tau +1)k-1}\\ \vdots\\ u_{\tau k}
\end{bmatrix}
= - \reachab_k^+ \bar{A} x_{\tau k}^{(2)} = 0
$ and $
\bar{u}_\tau =
\reachab_k
\begin{bmatrix}
u_{(\tau+1)k-1}\\ \vdots\\ u_{\tau k}
\end{bmatrix} = 0,
$
and consequently, for $\tau \geq \bar{\tau}$ and $\tau k \leq t < (\tau +1)k$ we also have $x_t^{(2)} = 0$, that is, $x_t^{(2)} = 0 \;\; \fa t \geq \bar{\tau}k$,
which proves (P2)
for the subsystem $\bigl(A_{22}, B_2\bigr)$ of our system \eqref{STABLE_SYSTEM}.

Finally, to extend the result (P2) to the general case (where $A=\diag(A_{11}, A_{22})$), it suffices to note
that, since for $t\geq \bar{\tau} k$ it also holds $u_t=0$, from the time $\bar{\tau} k$
onwards the subsystem $(A_{11}, B_1)$ is in open loop.
Since we imposed $A_{11}$ to be Schur stable, the state $x^{(1)}_t$ of the latter converges to zero
as $t\rightarrow\infty$.
This proves the theorem.
\end{proof}

\section{Numerical Example}
\label{SECTION_EXAMPLE}
An example follows, which shows that our nonlinear policy is readily computable, and 
effective in bounding the state of a stable linear system
in the mean square.
We executed $1000$ runs of simulation of the system $x_{t+1} = A x_t + B u_t + w_t$,
where $ 
A = \left[\begin{smallmatrix}
\cos{\varphi_1} & -\sin{\varphi_1} & 0 & 0 \\ 
\sin{\varphi_1} &  \cos{\varphi_1} & 0 & 0 \\ 
0               & 0                & 0.5 & 0\\ 
0               & 0                & 0 & 0.9
\end{smallmatrix}
\right]$, $
B = \left[\begin{smallmatrix}
1 \\
0 \\
0 \\
0 
\end{smallmatrix}\right]$, 
with $\varphi_1 = 0.8$, $x_0 = \bmat{cccc}10&20&30&40\emat^\top$, and where
$w_t$ is a Gaussian white noise with variance $I_4$.
This system is marginally stable and, as is easily seen, the $2$-dimensional
subsystem with eigenvalues on the unit circle is reachable in $2$ steps,
whereas the $2$-dimensional Schur-stable subsystem is not reachable at all.
The control authority $R$ was chosen approximately equal to $3.6$ according to a rough estimate
of $C_1 = \EE_x\bigl[\norm{w_t}\bigr]$.
It should be noticed that smaller values of $R$ are also sufficient to stabilize the system.
%
\begin{figure}[h]
\begin{center}
\vspace{-5mm}
\includegraphics[width=12cm]{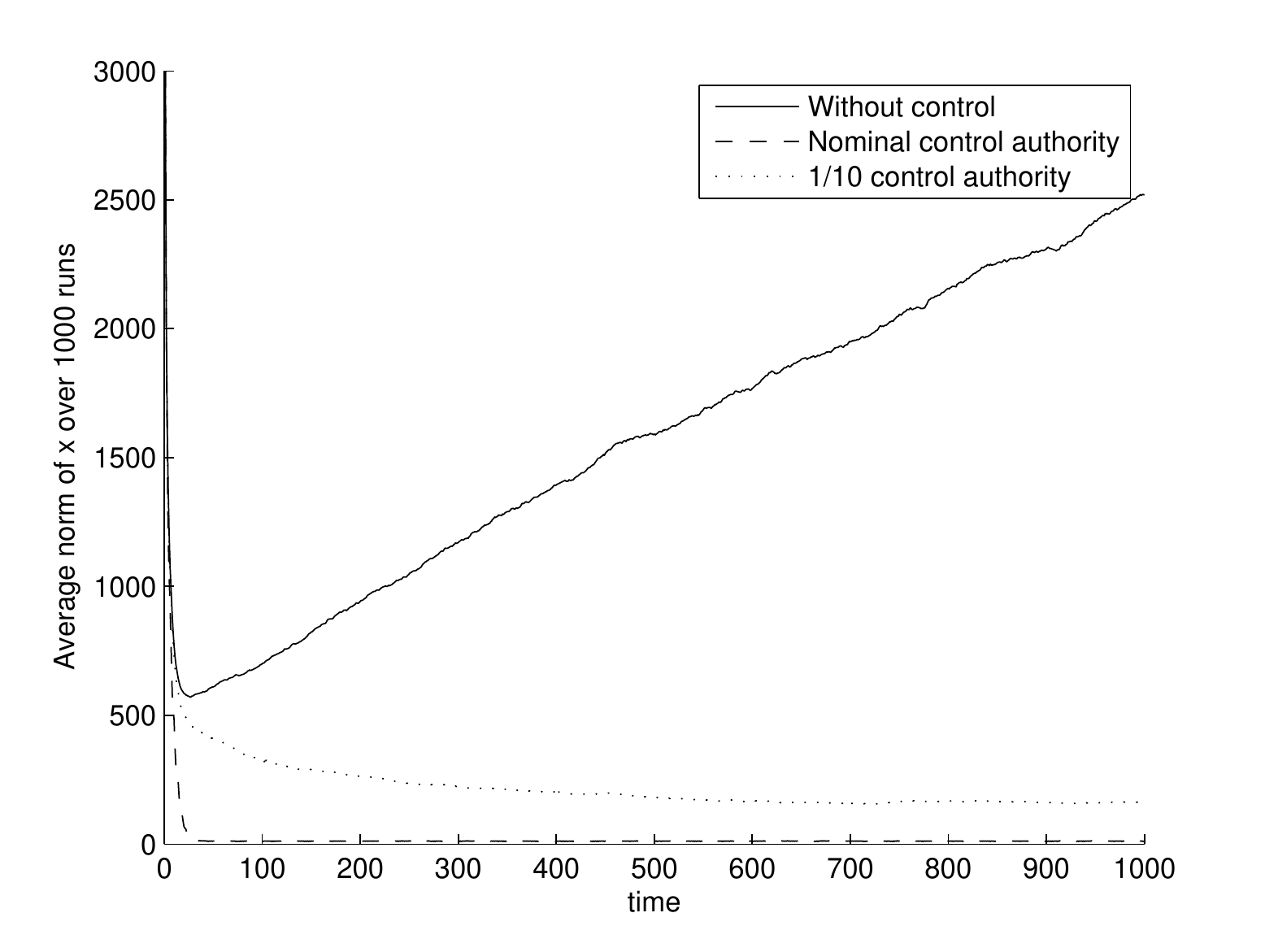}
\vspace{-10mm}
\caption{Empirical average of $||x_t||^2$ over $1000$ runs.}
\vspace{-5mm}
\label{fig:EXAMPLE}
\end{center}
\end{figure}
%
Figure \ref{fig:EXAMPLE} shows the empirical average of $||x_t||^2$ over the $1000$ runs,
respectively with disabled control, with the chosen control authority,
and with one tenth of the chosen control authority.
%

\section{A Conjecture}
\label{SECTION_CONCLUSION}

We conjecture that if the noise has bounded variance, then \emph{given any arbitrary positive uniform upper-bound} on the norm of the control, there exists a \emph{stationary feedback policy} such that the closed-loop system is mean-square bounded.  It appears to us that a proof of this conjecture will require substantially new and nontrivial techniques.



\end{document}